\documentclass[sn-mathphys-num]{sn-jnl}
\usepackage{mathtools,amssymb,latexsym,amsthm,amsmath,amsfonts,enumerate}  

\newcommand{\Li}{{\mbox{Lip}}}
\newcommand{\cS}{{\mathcal S}}
\newcommand{\cC}{{\mathcal C}}
\newcommand{\ux}{\underline{x}}
\newcommand{\uy}{\underline{y}}
\newcommand{\E}{\textbf{E}}
\newcommand{\cP}{\mathcal{P}}
\newcommand{\uz}{\underline{z}}
\newcommand{\uzz}{\underline{\zeta}}
\newcommand{\f}{\tilde{f}}
\newcommand{\g}{\tilde{g}}
\newcommand{\R}{\mathbb{R}}
\newcommand{\bC}{\textbf{C}}
\newcommand{\cD}{{\mathcal D}}

\newtheorem{theorem}{Theorem}[section]
\newtheorem{lemma}[theorem]{Lemma}
\newtheorem{remark}[theorem]{Remark}
\newtheorem{definition}{Definition}[section] 

\begin{document}

\title[Hardy decomposition of higher order Lipschitz classes by polymonogenic functions]{Hardy decomposition of higher order Lipschitz classes by polymonogenic functions}

\author*[1]{\fnm{Lianet} \sur{De la Cruz Toranzo}}\email{lianetcruzt@gmail.com}
\author[2,3]{\fnm{Ricardo} \sur{Abreu Blaya}}\email{rabreublaya@yahoo.es}
\author[1]{\fnm{Swanhild} \sur{Bernstein}}\email{swanhild.bernstein@math.tu-freiberg.de}

\affil*[1]{\orgdiv{Institut für Angewandte Analysis}, \orgname{TU Bergakademie Freiberg}, \orgaddress{\street{Pr\"uferstr. 9}, \city{Freiberg}, \postcode{09599}, \state{Sachsen}, \country{Germany}}}
\affil[2]{\orgdiv{Facultad de Matem\'aticas}, \orgname{Universidad Aut\'onoma de Guerrero}, \orgaddress{\street{Ave. L\'azaro C\'ardenas}, \city{Chilpancingo}, \postcode{39087}, \state{Guerrero}, \country{Mexico}}}
\affil[3]{\orgdiv{Profesor invitado}, \orgname{Universidad UTE Ecuador}, \orgaddress{
\country{Ecuador}}}

\abstract{
In this paper we find a decomposition of higher order Lipschitz functions into the traces of a polymonogenic function and solve a related Riemann-Hilbert problem. Our approach lies in using a cliffordian Cauchy-type operator, which behaves as an involution operator on higher order Lipschitz spaces. The result obtained is a multidimensional sharpened version of the Hardy decomposition of H\"older continuous functions on a simple closed curve in the complex plane.}



\maketitle
\section{Introduction}\label{sec:intro}
The space of H\"older functions of exponent $0<\alpha\leq 1$ defined on a set of the complex plane, has numerous properties and plays an important role in the theory of analytic functions. For instance, it makes sense to speak of a singular version of the Cauchy transform. Actually, by the Plemelj-Sokhotski's formulae, the boundary value problem on finding a sectionally analytic function vanishing at infinity and with jump on a closed smooth curve equals to a given H\"older function $f$, is solved by the Cauchy transform of $f$ \cite{Mu}. This result leads to the Hardy decomposition of the H\"older class as the sum of two disjoint H\"older spaces, which are characterized for the limiting values of the Cauchy transform as approaching the curve from the interior and the exterior domain.
When $\alpha=1$, the H\"older condition becomes the well-known Lipschitz condition, which has been widely used in the theory of differential equations, whereas when $\alpha>1$ the class of H\"older is constituted by constants only. However, a non-trivial space emerges under that condition if we consider the higher order Lipschitz functions. In \cite{AD} the authors found a decomposition of the Lipschitz class in the context of polyanalytic functions \cite{Ba}, that is, null solutions of the iterated Cauchy-Riemann equation. When looked at from a more general point of view, the null solutions to the iterated Dirac operator in Clifford analysis (see for instance,  \cite{BDS} and \cite{GS}), leads to polymonogenic functions and therefore a natural question to ask is whether such a decomposition continues to hold in this multidimensional setting. Very recently in \cite{DAB2}, while attempting to prove it, the authors obtained a decomposition of the first order Lipschitz class into traces of bimonogenic functions, which are nothing more than Clifford algebra-valued harmonic functions as the Dirac operator factorize the Laplacian. In this paper we bring together the methods carried out in the above mentioned works and prove a general decomposition of higher order Lipschitz classes into Clifford algebra-valued polymonogenic functions.
The organization of the paper is the following. In Section \ref{Sec2} we present some preliminaries about real and Clifford analysis. In Section \ref{Sec3} we summarize without proofs the relevant material on the Cauchy type integral operator and its singular version related to higher order Lipschitz functions and we extend Theorem 3 of \cite{AA} to the general order with the aid of a multi-index formula for the iterated Dirac operator. In Section \ref{sec_main} we state and prove our main results. Finally, we also solve a Riemann Hilbert problem for polymonogenic functions with given Lipschitz data. 

\section{Definitions and related results}\label{Sec2}
An $m$-tuple $j= (j_1,\dots ,j_m)$ with non-negative integer components $j_i$ for every $i=\overline{1,m}$ is called multi-index. Given a multi-index $j$ we recall: $x^j= x_1^{j_1}\cdots x_m^{j_m}$, $j!=j_1!\cdots j_m!$, $|j|=j_1+\cdots +j_m$ and $\partial^{j}=\partial^{j_1}_{ x_1}\dots \partial^{j_m}_{ x_m}.$ The reader is cautioned that, throughout this work, the letters $j$, $l$, $p$ and $q$ will be reserved for multi-indexes, while the letters $i$, $s$, $u$ and $v$ are reserved for indexes of summation.
\begin{definition}
Let $\E$ be a closed subset of $\R^m$, $k$ a non-negative integer and $0<\alpha\le 1$. We shall say that a real valued function $f$, defined on $\E$, belongs to the higher order Lipschitz class $\Li(k+\alpha,\E)$ if there exist real valued functions  $f^{(j)}$, $0<|j|\le k$, defined on $\E$, with $f^{(0)}=f$, which together with
\begin{equation}\label{L1}
		R_j(x,y)=f^{(j)}(x)-\sum_{|j+l|\le k}\frac{f^{(j+l)}(y)}{l!}(x-y)^{l},\quad x, y\in\E
\end{equation}
satisfy
\begin{equation}\label{L2}
		|f^{(j)}(x)|\le M,\quad|R_j(x,y)|\le M|x-y|^{k+\alpha-|j|},\quad x, y\in\E, |j|\le k,
\end{equation}
where here and subsequently $M$ is a positive constant not necessary the same at different occurrences.
\end{definition}
\begin{remark} We observe that the function $f^{(0)}$ does not necessarily determine the elements $f^{(j)}$ (see \cite[p. 176]{St}). Therefore, an element of $\Li(k+\alpha,\E)$ should be interpreted as a collection $f=\{f^{(j)}:\E\mapsto\R,\quad|j|\le k\}$. If no confusion can arise, we write $f$ for both the collection and the element $f^{(0)}$. We use superscripts to distinguish between individual elements and sub-collections. In this way, for all $|j|\leq k$, the data $f_{(j)}:=\{f^{(j+l)},\quad0\le|l|\le k-|j|\}$ belongs to $\Li(k+\alpha-|j|,\E)$. In particular, $f_{(j)}=\{f^{(j)}, \quad|j|=k\}$ belongs to the H\"older class $\bC^{\alpha}(\E)$. When $\E=\R^m$, however, the functions $f^{(j)}$ are uniquely determined by $f$ and $\Li(k+\alpha,\R^m)$ actually consists of continuous and bounded functions $f$ with continuous and bounded partial derivatives $\partial^{j}f=f^{(j)}$ up to the order $k$. In what follows, we shall concern ourselves only with case $0<\alpha<1.$
\end{remark}

The higher order Lipschitz class introduced by E. Stein in \cite{St} appears to be the most appropriate class of functions to describe properties in terms of Banach spaces of certain linear extension operators related to Whitney-type problems on extensions of real functions defined on closed sets of $\R^m$. For the convenience of the reader, we include the Whitney extension theorem \cite{Wh} as it will be a key point to define our multidimensional singular integral operator later. 

\begin{theorem}(Whitney's theorem)\label{thm_whit}
Let $f\in\Li(k+\alpha, \E)$.  Then, there exists a function $\f\in\Li(k+\alpha,\R^m)$ satisfying
\begin{enumerate}[(i)]
\item $\f|_\E = f^{(0)},\quad\partial^{j}\f|_\E=f^{(j)}$, $0<|j|\le k$,
\item $\f\in C^\infty(\R^{m} \setminus \E)$,
\item $|\partial^{j}\f(x) | \leqslant M \quad \mbox{\em dist}(x,\E)^{\alpha-1}$, for $|j|=k+1$ and $x\in\R^{m}\setminus\E$.
\end{enumerate}
\end{theorem}

Actually, the linear extension operator 
\[
{\mathcal{E}}_k(f^{(j)})(x)=\begin{cases}
	f^{(0)}(x),&x\in\E\\ {\sum\limits_{i}}'P(x,p_i)\varphi_i^*(x),&x\in\E^c,
\end{cases}
\]
where $\sum_i\varphi_i^*=1$ and $P(x,y)$ denotes the Taylor expansion of $f$ about $y$, maps $\Li(k+\alpha,\E)$ continuously into $\Li(k+\alpha,\R^m)$ and the norm is independent of the closed set $\E$. For a deeper discussion we refer the reader to \cite[Ch. VI]{St}.

Although the preceding definition involves real-valued functions, we now introduce the algebras we shall be concerned with.  

\begin{definition}
Let us consider an orthonormal basis $\{e_i\}_{i=1}^m$ of $\R^m$ governed by the multiplication rules: $e_i^2=-1$, $e_ie_{j}=-e_{j}e_i, i\neq j$ for every $i,j=1,\dots, m.$ The real Clifford algebra $\R_{0,m}$ is that generated by $\{e_i\}_{i=1}^m$ over $\R$.
\end{definition}

An element $a\in\R_{0,m}$ may be written as $a=\sum_{A} a_A e_A,$ where $a_A\in\R$ and $A$ runs over all the possible ordered sets $A=\{1\le i_1<\cdots < i_k\le m\}$ or $A=\emptyset$ and $e_A:=e_{i_1}e_{i_2}\cdots e_{i_k},\quad e_0=e_\emptyset=1.$ It is clear that $\R^m$ is embedded in $\R_{0,m}$, provided we identify $\ux=(x_1,\ldots,x_m)\in\R^m$ by $\ux=x_1e_1+\dots +x_me_m;\quad x_i\in\R,i=\overline{1,m}$. A norm for $a\in \R_{0,m}$ may be defined by $\|a\|^{2}=\sum_A|a_A|^{2}.$ In particular, $\|\ux\|=|\ux|$ for $\ux\in\R^m$, where $|\cdot|$ stands for the Euclidean norm.

We shall consider Clifford algebra-valued functions defined on subsets of $\R^m$, namely $f=\sum_{A} f_A e_A$, $f_A:\R^m\to\R$. Conditions as continuity, differentiability, Lipschitz and so on, are ascribed to an $\R_{0,m}$-valued function by doing so to each of its real components $f_A$. From now on, $\Omega$ stands for a Jordan domain, i.e. a bounded oriented connected open subset of $\R^{m}$ whose boundary $\Gamma$ is a compact topological surface. For simplicity, we assume $\Gamma$ to be sufficiently smooth, e.g. Lyapunov surface. In addition, we denote the interior of $\Omega$ by $\Omega_+$ and the exterior by $\Omega_-$.

\begin{definition}
An $\R_{0,m}$-valued function $f$ in $\bC(\Omega)$ is called left (resp. right) monogenic if $\cD_{\ux}\, f = 0$ (resp. $f\, \cD_{\ux} = 0$) in $\Omega$, where $\cD_{\ux}$ is the Dirac operator
\[
	\cD_{\ux}={\partial_{x_1}}e_1+{\partial_{x_2}}e_2+\cdots +{\partial_{x_m}}e_m.
\]
\end{definition}

The so-called Clifford-Cauchy kernel 
\[
E_0 (\ux)=-\frac{1}{\sigma_{m}}\frac{\ux}{|\ux|^{m}}\quad(\ux\neq0),
\] 
is a two-sided monogenic function. Here $\sigma_{m}$ stands for the surface area of the unit sphere in $\R^{m}$.

\begin{definition}
An $\R_{0,m}$-valued function $f$ in $\bC^k(\Omega)$ is called (left) polymonogenic of order $k$ or simply $k$-monogenic if $\cD_{\ux}^k\, f = 0$ in $\Omega$, where $\cD_{\ux}^k$ is the iterated Dirac operator.
\end{definition}

Every polymonogenic function can be represented in terms of its values and those of its successive derivatives on the boundary \cite[Thm 7]{Ry}. Namely, if $f$ is $k$-monogenic in $\Omega$, then
\begin{equation}\label{repr_dom_+}
f(\ux)=\sum_{u=0}^{k-1}\int_\Gamma (-1)^uE_u(\uy-\ux)n(\uy) \cD_{\uy}^uf(\uy)d\uy, x\in\Omega,
\end{equation}
where the $E_u$'s satisfy $\partial_{\ux} E_{u+1}(\ux)=E_{u}(\ux)$ for all $\ux\not=0$ and $u=0,\ldots, k-1$.

\section{Cauchy-type and singular integral operators}\label{Sec3}

Formula \eqref{repr_dom_+} and Whitney extension theorem (component-wise applied to $\R_{0,m}$-valued functions) can be combined to give the following operators, which have been shown to be well-defined \cite{DAB1}: 
\begin{definition}\label{def_op} 
The Cauchy transform related to $f\in\Li(k+\alpha,\Gamma)$ and its singular version are given respectively by
\begin{align}\label{def_CT_k}	
		\cC_k^{(0)}f=[\cC_kf]^{(0)}(\ux)=\sum\limits_{s=0}^{k}\int_\Gamma (-1)^sE_s(\uy-\ux)n(\uy)\cD_{\uy}^s\f(\uy)d\uy,\quad\ux\in \Omega\setminus\Gamma,\\
		\label{def_Sk_0}
		\cS_k^{(0)}f=[\cS_kf]^{(0)}(\uz)=2\sum\limits_{s=0}^{k}\int_\Gamma (-1)^sE_s(\uy-\uz)n(\uy)\cD_{\uy}^s\f(\uy)d\uy,\quad\uz\in\Gamma,
\end{align}
where $\f$ denotes the $\R_{0,m}$-valued Whitney extension of $f\in\Li(k+\alpha,\Gamma)$. 
\end{definition}

We see at once that our definition coincides with the classical one related to a H\"older function $f\in\bC^{\alpha}(\Gamma)$. Namely, 
\begin{align}
	\cC_0 f(\ux)&=\int_\Gamma E_0(\uy-\ux)n(\uy)f(\uy)d\uy,\quad\ux\in\R^m\setminus\Gamma, \label{def_CT_0}\\
	\cS_0 f(\uz)&=2\int_\Gamma E_0(\uy-\uz)n(\uy)f(\uy)d\uy,\quad\uz\in\Gamma. \label{def_S_0}
\end{align}

The Cauchy transform \eqref{def_CT_0} represents a monogenic function in $\R^m\setminus\Gamma$ with continuous limiting values as approaching the contour $\Gamma$. Actually, the Plemelj-Sokhotski formula \cite{Ift} gives an even more precise information:
\begin{equation}\label{cliff_Plemelj-Sokhotski_C0}
	[\cC_0f]^+=\frac{1}{2}[I+\cS_0]f\quad \mbox{and }\quad [\cC_0f]^-=\frac{1}{2}[-I+\cS_0]f,
\end{equation}
where $I$ is the identity operator and $\cC_0^\pm f(\uz)=\lim\limits_{\substack{\ux\to \uz\\ \ux\in\Omega_\pm}}\cC_0f(\ux).$

When looked at from a general point of view, the preceding fact suggests that identical conclusions hold with respect to $\cC_k^{(0)}$ and $\cS_k^{(0)}$. In fact, as it was pointed out in \cite{DAB1}, the function $\cC_k^{(0)}$ is polymonogenic of order $k+1$ and 
\[
\begin{cases}
	[\cC_k^{(0)}f]^+(\uz)-[\cC_k^{(0)}f]^-(\uz)=\f\mid_\Gamma=f^{(0)}(\uz),&\quad\uz\in\Gamma,\\
	[\cC_k^{(0)}f]^+(\uz)+[\cC_k^{(0)}f]^-(\uz)=\cS_k^{(0)}f(\uz),&\quad\uz\in\Gamma, 
\end{cases}
\]
or, equivalently,
\begin{equation}\label{cliff_Plemelj-Sokhotski_Ck}
	[\cC_k^{(0)}f]^+=\frac{1}{2}[I^{(0)}+\cS_k^{(0)}]f\quad \mbox{and }\quad [\cC_k^{(0)}f]^-=\frac{1}{2}[-I^{(0)}+\cS_k^{(0)}]f.
\end{equation}
Here $I^{(0)}f=f^{(0)}$ and, as usual, $[\cC_k^{(0)}f]^\pm(\uz)=\lim\limits_{\substack{\ux\to \uz\\ \ux\in\Omega_\pm}}\cC_k^{(0)}f(\ux).$

A close inspection of Definition \ref{def_op} makes it natural to try to relate the Lipschitz data $f$ to the iterated Dirac operator of the Whitney extension $\f$. For this purpose, we first show a formula for the iterated Dirac operator in terms of multi-indices.
\begin{lemma}\label{D_multii}
Let $e_{(l)}=e_i$ provided $|l|=1$ and $l_i=1$. Then
\begin{equation}\label{eq_Ds}
	\cD_{\ux}^s= c_s
\begin{cases}
			\sum\limits_{\substack{|j|=s\\ j_i-even}} \partial^{(j)}_{\ux}, & \quad s\quad\emph{even},\\
			\sum\limits_{\substack{|j+l|=s\\j_i-even\\|l|=1}}e_{(l)}\partial^{(j+l)}_{\ux}, &\quad s\quad\emph{odd},
\end{cases}	
\end{equation}
where $c_s=1$ if $s\equiv 0,1 \mod 4$ and $c_s=-1$ if $s\equiv 2,3 \mod 4$.
\end{lemma}	
\begin{proof} The proof is by induction on $s$. For $s=1$, $\cD_{\ux}=\sum\limits_{i=1}^{m}e_i\partial_{x_i}=\sum\limits_{|l|=1}e_{(l)}\partial_{\ux}^{(l)}$.
Setting $j=0$, we can easily check the formula. 
Assuming \eqref{eq_Ds} to hold for $s$, we will prove it for $s+1$. If $s+1$ is odd, $s$ is even and then, by the induction hypothesis, $\cD_{\ux}^s= c_s\sum\limits_{\substack{|j|=s\\ j_i-even} }\partial^{(j)}_{\ux}$ being $s=4p$ or $4p+2$. After applying the Dirac operator one more time, we get
\[
	\cD_{\ux}(\cD_{\ux}^s)=\cD_{\ux}^{s+1}= \sum\limits_{|l|=1}e_{(l)} \partial^{(l)}_{\ux}\Big(c_s\sum\limits_{\substack{|j|=s\\ j_i-even}} \partial^{(j)}_{\ux}\Big)=c_s\sum\limits_{\substack{|j+l|=s+1\\j_i-even\\|l|=1}}e_{(l)}\partial^{(j+l)}_{\ux}.
\]
Since $c_s=c_{s+1}$ when $s$ is even, \eqref{eq_Ds} holds for $s+1$ odd. 
We now suppose that $s+1$ is even, so that $s=4p+1$ or $4p+3$ and
\begin{align*}
		\cD_{\ux}^{s+1}&= \sum\limits_{|q|=1}e_{(q)} \partial^{(q)}_{\ux}\Big(c_s\sum\limits_{\substack{|j+l|=s\\j_i-even\\|l|=1}}e_{(l)} \partial^{(j+l)}_{\ux}\Big)=c_s\sum\limits_{\substack{|j+l+q|=s+1\\j_i-even\\|l|=1, |q|=1}}e_{(q)}e_{(l)} \partial^{(j+l+q)}_{\ux}\\
		&=-c_s\sum\limits_{\substack{|j+l+l|=s+1\\j_i-even\\|l|=1}} \partial^{(j+l+l)}_{\ux},
\end{align*}
the last equality being a consequence of the fact that $e_{(q)}e_{(l)}+e_{(l)}e_{(q)}=0$ for every $l\neq q$, $|l|=|q|=1$ and $e_{(l)}^2=-1.$ Since $p=j+l+l$ has only even entries and  $c_s=-c_{s+1}$ when $s$ is odd, \eqref{eq_Ds} holds for $s+1$ even.
\end{proof}
Lemma \ref{D_multii} provides the following reformulation of the singular operator \eqref{def_Sk_0} in terms of the Lipschitz data:
\begin{align}\label{redef_Sk}
	\cS_k^{(0)}f(\uz)&=\sum\limits_{\substack{s-even\\ s\leq k}}\int_\Gamma E_s(\uy-\uz)n(\uy)\sum\limits_{\substack{|j|=s\\ j_i-even}} f^{(j)}(\uy)d\uy \notag\\
	&-\sum\limits_{\substack{s-odd\\ s\leq k}}\int_\Gamma E_s(\uy-\uz)n(\uy)\sum\limits_{\substack{|j+l|=s\\j_i-even\\|l|=1}}e_{(l)} f^{(j+l)}(\uy)d\uy, \quad\uz\in \Gamma.
\end{align}
The corresponding Cauchy transform can be rewritten in the same manner, but for our purpose here, the focus will be on the singular operator. At first glance, \eqref{redef_Sk} does not even appear to be an injection as it is defined only by some elements of the Lipschitz data $f:=\{f^{(j)},\quad|j|\le k\}$. However, even though a Lipschitz data is not determined by none of its proper sub-collections, our following generalization of Theorem 3 in \cite{AA} states that, in the case of a compact set without isolated points, the Lipschitz data $f$ is determined by a special combination of some of its elements. 
\begin{theorem}\label{thm_fj}
Let $f\in\Li(k+\alpha,\Gamma)$ such that $f^{(0)}\equiv 0$ and for $1\leq s\leq k$ it satisfies 
\begin{align}
		\sum\limits_{\substack{|j|=s\\ j_i-even}} f^{(j)}=0, \quad&s \,even,\label{hip1}\\
		\sum\limits_{\substack{|j+l|=s\\j_i-even\\|l|=1}}e_{(l)} f^{(j+l)}=0,\quad &s \,odd.\label{hip2}
\end{align}
Then $f\equiv 0$, i.e. $f^{(j)}=0$ for all $|j|\leq k.$ 
\end{theorem}

\begin{proof}
The proof is by induction on $k$. The case $k=1$ has already been proved in \cite[Thm 3]{AA}. Assume the theorem holds for $k$. Let $f\in\Li(k+1+\alpha,\Gamma)$ such that for every $1\leq s\leq k+1$ it satisfies \eqref{hip1} and \eqref{hip2}.
	
Define $F^{(0)}=\sum\limits_{|j|=1}e_{(j)} f^{(j)} $ and $F^{(l)}=\sum\limits_{|j|=1}e_{(j)} f^{(j+l)}$. Then $\{F^{(l)},|l|\leq k\}\in\Li(k+\alpha,\Gamma)$ and $F^{(0)}=0$ by \eqref{hip2}. For $2\leq s\leq k$ even, we get
\[
	\sum\limits_{\substack{|l|=s\\ l_i-even}} F^{(l)}=\sum\limits_{\substack{|l|=s\\ l_i-even}} \Big(\sum\limits_{|j|=1}e_{(j)} f^{(j+l)}\Big)=\sum\limits_{\substack{|j+l|=s+1\\l_i-even\\|j|=1}} e_{(j)}f^{(j+l)} =0\quad \text{by}\quad\eqref{hip2}.
\]
On the other hand, if $1\leq s\leq k$ is odd, we apply again the multiplication rules for every $l\neq q$, $|l|=|q|=1:$ $e_{(q)}e_{(l)}+e_{(l)}e_{(q)}=0$ and $e_{(l)}^2=-1$ wich together \eqref{hip1} yield
\[
	\sum\limits_{\substack{|q+l|=s\\q_i-even\\|l|=1}}e_{(l)} F^{(q+l)}=\sum\limits_{\substack{|q+l|=s\\q_i-even\\|l|=1}}e_{(l)}\Big(\sum\limits_{|j|=1}e_{(j)} f^{(j+q+l)}\Big)=-\sum\limits_{\substack{|j^*|=s+1\\j^*_i-even}} f^{(j^*)}=0.
\]
Thus the induction hypothesis implies that $F^{(l)}\equiv 0$ for all $|l|\leq k.$ Set now $G_1^{(0)}=f^{(1,0,\ldots,0)}$ and $G_1^{(l)}=f^{(1,0,\ldots,0)+l}$. It is obvious that $G_1=\{G_1^{(l)}, |l|\leq k\}\in\Li(k+\alpha,\Gamma)$. We can now proceed analogously to the proof of \cite[Thm 3]{AA} to conclude that $f^{(j)}=0$ whenever $|j|=1$ and so $G_1^{(0)}=0$. Let us verify the other induction assumptions. Start by considering $2\leq s\leq k$ even. If this is so, 
\begin{align*}
		\sum\limits_{\substack{|l|=s\\ l_i-even}} G_1^{(l)}&=\sum\limits_{\substack{|l|=s\\ l_i-even}} f^{(1,0,\ldots,0)+l}\\
		&=-\sum\limits_{\substack{|q+l^*|=s-1\\q_i-even\\|l^*|=1}}e_{(l^*)}\Big(\sum\limits_{|j|=1}e_{(j)}f^{j+(1,0,\ldots,0)+q+L}\Big)\\
		&=-\sum\limits_{\substack{|q+l^*|=s-1\\q_i-even\\|l^*|=1}}e_{(l^*)}\Big(\sum\limits_{|j|=1}e_{(j)}f^{(j+p)}\Big),
\end{align*}
where $p=(1,0,\ldots,0)+q+l^*$, $|p|\leq k$. We note that $\sum\limits_{|j|=1}e_{(j)}f^{(j+p)}=F^{(p)}$ for $|p|\leq k$, so that the above expression vanishes. In a similar way, for $1\leq s\leq k$ odd we have
\begin{align*}
		\sum\limits_{\substack{|j+l|=s\\j_i-even\\|l|=1}}e_{(l)}G_1^{(j+l)}&=\sum\limits_{\substack{|j+l|=s\\j_i-even\\|l|=1}}e_{(l)}f^{(1,0,\ldots,0)+j+l}\\
		&=\sum\limits_{\substack{|j|=s-1\\ j_i-even}} \Big(\sum\limits_{|l|=1}e_{(l)}f^{(1,0,\ldots,0)+j+l}\Big)\\
		&=\sum\limits_{\substack{|j|=s-1\\ j_i-even}}\Big(\sum\limits_{|l|=1}e_{(l)}f^{(l+p)}\Big),
\end{align*}
where $p=(1,0,\ldots,0)+j$, $|p|\leq k$. Now the induction hypothesis implies that $G_1^{(l)}=0$ for all $|l|\leq k$. To complete the proof, similar considerations apply to the functions $G_i=f^{(0,\ldots,1,\ldots,0)}$, where the multi-index has all entries zero except the $i$th which is one. Consequently, $f^{(j)}=0$ for all $|j|\leq k+1.$
\end{proof}

For completeness of exposition, we conclude this section with the relevant material from \cite{DAB1}. Let us first introduce the notation $E_{u}^{(j)}(\ux)=\partial_{\ux}^{j} E_{u}(\ux),\quad\ux\not=0$. We define
\[
\cS_k^{(j)} f(\uz)=2\sum\limits_{u=0}^k\int_\Gamma (-1)^uE_u^{(j)}(\uy-\uz)n(\uy)\cD_{\uy}^u R(\uy,\uz)d\uy+f^{(j)}(\uz)
\]
to be the associated elements of the collection $\cS_k f$. Although the definition may seem artificial, it is actually very much in the spirit of the classical case. Indeed, when $k=0$ it reduces to
\[
\cS_0 f(\uz)=2\int_\Gamma E_0(\uy-\uz)n(\uy)[f(\uy)-f(\uz)]d\uy+f(\uz),
\]
which is nothing more than a reformulation of \eqref{def_S_0}. 

The Plemelj-Privalov theorem in this framework asserts that, for $0<\alpha<1$ the inclusion $\cS_k (\Li(k+\alpha,\Gamma))\subset\Li(k+\alpha,\Gamma)$ holds, and thus, makes it legitimate to ask whether the operator is an involution. Whereas for a single function $f\in\bC^\alpha(\Gamma)$ the equality $\cS_0^2=I$ makes sense, it is somewhat misleading for a Lipschitz data. The involution, to be discussed in Section \ref{sec_main}, must thus be approached by considering every element, i.e. $[\cS_k^2f]^{(j)}=f^{(j)}$ for all $|j|\leq k$.

\section{Involution property and consequences}\label{sec_main}

This section is devoted to establishing and proving our main result, but first we need a key lemma. 
\begin{lemma}\label{lemmaSj}
For $s$ even (resp. odd), the sum 
\[
	c_s\sum\limits_{\substack{|j|=s\geq 2\\ j_i-even}}[\cS_kf]^{(j)}(\uy) \Big(\mbox{resp.} c_s\sum\limits_{\substack{|j+l|=s\\ j_i-even,|l|=1}}e_{(l)}[\cS_kf]^{(j+l)}(\uy)\Big)
\]
is equal to
\begin{equation}\label{lemaSj_a}
		2\sum\limits_{u=1}^k
		\int_\Gamma (-1)^u\cD_{\uy}^sE_u(\uzz-\uy)n(\uzz)\cD_{\uzz}^u \f(\uzz)d\uzz
\end{equation}
or equivalently, to
\begin{equation}\label{lemaSj_b}
		2\sum\limits_{u=0}^{k-s}
		\int_\Gamma (-1)^uE_u(\uzz-\uy)n(\uzz)\cD_{\uzz}^{u+s} \f(\uzz)d\uzz.
\end{equation}
\end{lemma}
\begin{proof}
Indeed, if $s$ is even 
\begin{align}\label{eq_Sj_even}
		&c_s\sum\limits_{\substack{|j|=s\geq 2\\ j_i-even}}[\cS_kf]^{(j)}(\uy)\notag\\
		&=c_s\sum\limits_{\substack{|j|=s\geq 2\\ j_i-even}}\Bigl\{2\sum\limits_{u=0}^{k}\int_\Gamma(-1)^uE_u^{(j)}(\uzz-\uy)n(\uzz)\cD_{\uzz}^uR(\uzz,\uy)d\uzz+f^{(j)}(\uy)\Bigr\}\notag\\
		&=\sum\limits_{u=1}^k
		2\int_\Gamma (-1)^u\cD_{\uy}^sE_u(\uzz-\uy)n(\uzz)\cD_{\uzz}^uR(\uzz,\uy)d\uzz
		+c_s\sum\limits_{\substack{|j|=s\geq 2\\ j_i-even}}f^{(j)}(\uy).
\end{align}
With Theorem \ref{thm_whit} in mind, we get 
\[
	\f(\uzz)=\f(\uy)+\sum\limits_{0<|l|\leq k}\frac{\partial_{\uy}^{(l)} \f(\uy)}{l!}(\uzz-\uy)^l+R(\uzz,\uy),\quad\uy, \uzz\in\Gamma.
\]
Therefore,
\begin{equation}\label{eq_Du_R}
		\cD_{\uzz}^u R(\uzz,\uy)
		=\cD_{\uzz}^u \f(\uzz)-\cD_{\uy}^u \f(\uy)-P_u[f](\uzz,\uy),\quad 0\leq u\leq k,
\end{equation}
where
\begin{align}\label{P_u}
		P_u[f](\uzz,\uy)=&\notag\\
		&c_u
\begin{cases}
			\sum\limits_{\substack{|p|=u\\ p_i-even}}\Bigl\{ \sum\limits_{|p|<|l|\leq k}\dfrac{\partial_{\uy}^{(l)} \f(\uy)}{(l-p)!}(\uzz-\uy)^{l-p} \Bigl\}, &u\, even,\\
			\sum\limits_{\substack{|p+q|=u\\p_i-even,|q|=1}}e_{(q)}\Bigl\{\sum\limits_{|p+q|<|l|\leq k}\dfrac{\partial_{\uy}^{(l)} \f(\uy)}{(l-(p+q))!}(\uzz-\uy)^{l-(p+q)}\Bigl\}, &u\, odd.
\end{cases}	
\end{align}
Let us note that, by definition, $P_k[f](\uzz,\uy)=0$. In this way, \eqref{eq_Du_R} becomes 
\[
	\cD_{\uzz}^k R(\uzz,\uy)=\cD_{\uzz}^k \f(\uzz)-\cD_{\uy}^k \f(\uy), u=k.
\] 
On the other hand, it is a matter of straightforward computation to show that
\begin{equation}
		\cD_{\uzz}^uP_s[f](\uzz,\uy)=P_{u+s}[f](\uzz,\uy)+\cD_{\uy}^{u+s} \f(\uy),\quad\forall u=1,\ldots,k-s \label{P_a}.
\end{equation}
Note that, for $u=k-s$, \eqref{P_a} becomes
\begin{equation}
		\cD_{\uzz}^{k-s}P_s[f](\uzz,\uy)=\cD_{\uy}^{k} \f(\uy).\label{P_b}
\end{equation}
Now, substituting \eqref{eq_Du_R} into \eqref{eq_Sj_even} gives
\begin{align}\label{eq_cs}
		c_s\sum\limits_{\substack{|j|=s\geq 2\\ j_i-even}}	[\cS_kf]^{(j)}(\uy)
		&=2\sum\limits_{u=1}^k
		\int_\Gamma (-1)^u\cD_{\uy}^sE_u(\uzz-\uy)n(\uzz)\cD_{\uzz}^u \f(\uzz)d\uzz\notag\\
		&-2\sum\limits_{u=1}^k
		\int_\Gamma (-1)^u\cD_{\uy}^sE_u(\uzz-\uy)n(\uzz)\cD_{\uy}^u \f(\uy)d\uzz\notag\\
		&-2\sum\limits_{u=1}^{k-1}
		\int_\Gamma (-1)^u\cD_{\uy}^sE_u(\uzz-\uy)n(\uzz)P_u[f](\uzz,\uy)d\uzz\notag\\
		&+c_s\sum\limits_{\substack{|j|=s\geq 2\\ j_i-even}}f^{(j)}(\uy).
\end{align}
Consequently, it suffices to show that the sum, say $``K"$, of the last three terms in \eqref{eq_cs} vanishes. 
	
By the kernel's rules 
\begin{align*}
		K(\uy)&=-2\int_\Gamma E_0(\uzz-\uy)n(\uzz)\cD_{\uy}^{s} \f(\uy)d\uzz\\
		&-2\sum\limits_{u=1}^{k-s}\int_\Gamma (-1)^{u} E_{u}(\uzz-\uy)n(\uzz)\Big[\cD_{\uy}^{u+s}\f(\uy)+P_{u+s}[f](\uzz,\uy)\Big]d\uzz\notag\\
		&- 2\int_\Gamma  E_0(\uzz-\uy)n(\uzz)P_{s}[f](\uzz,\uy)d\uzz\notag\\
		&+c_s\sum\limits_{\substack{|j|=s\geq 2\\ j_i-even}}f^{(j)}(\uy).
\end{align*}
By \eqref{P_a}, \eqref{P_b}, Theorem \ref{thm_whit} and  the fact that $\int_\Gamma E_0(\uzz-\uy)n(\uzz) d\uzz=\frac{1}{2}$, we get 
\[
	K(\uy)=	-2\sum\limits_{u=0}^{k-s}\int_\Gamma (-1)^{u} E_{u}(\uzz-\uy)n(\uzz)\cD_{\uzz}^{u}P_{s}[f](\uzz,\uy)d\uzz.
\]
Set $F(\uzz)=P_s[f](\uzz,\uy)$ and for $\epsilon>0$ sufficiently small let $B_\epsilon(\uy)$ be the open ball with center in $\uy\in\Gamma$ and radius $\epsilon$, $C_{\epsilon}(\uy)$ its boundary and $\Gamma_{\epsilon}:=\Gamma\setminus\Gamma\cap B_\epsilon(\uy)$. We thus get
\[
	p.v. \int_\Gamma E_u(\uzz-\uy)n(\uzz)\cD_{\uzz}^u F(\uzz)d\uzz=\lim_{\epsilon\to 0}\int_{\Gamma_\epsilon} E_u(\uzz-\uy)n(\uzz)\cD_{\uzz}^u F(\uzz)d\uzz.
\]
Denote by $\Omega_*$ the interior of $\Gamma^*:=\Gamma_{\epsilon}\cup C_{\epsilon}^*(\uy)$, where $C_{\epsilon}^*=C_{\epsilon}\cap\Omega_+$. Then  $F(\uzz)$ is $(k-s+1)$-monogenic in $\Omega_*$ and continuous on $\Gamma^*\cup\overline{\Omega}_*$ so that the Cauchy formula (applied in $\R^m\setminus\overline{\Omega}_*$) yields
\[
	\sum\limits_{u=0}^{k-s}\int_{\Gamma^*} (-1)^{u} E_{u}(\uzz-\uy)n(\uzz)\cD_{\uzz}^{u}P_{s}[f](\uzz,\uy)d\uzz=0.
\]
We are left with the task of proving 
\[
	\lim\limits_{\epsilon\to 0}
	\Bigg\{
	\sum\limits_{u=0}^{k-s}\int_{C^*_{\epsilon}(\uy)} (-1)^{u} E_{u}(\uzz-\uy)n(\uzz)\cD_{\uzz}^{u}F(\uzz)d\uzz\Bigg\}=0.
\]
By \eqref{P_u}, $|P_s[f](\uzz,\uy)|\leq M \sum\limits_{0<|j|\leq k-s}|\uzz-\uy|^{|j|}$. Since $|j|\geq 1$, it holds for $|\uzz-\uy|$ sufficiently small that $|\uzz-\uy|^{|j|}\leq |\uzz-\uy|$. Therefore, for $u=0$
\[
	\Big|\int_{C^*_{\epsilon}(\uy)} E_0(\uzz-\uy)n(\uzz)P_s[f](\uzz,\uy)d\uzz\Big|\leq M
	\int_{C^*_{\epsilon}(\uy)}\frac{|\uzz-\uy|}{|\uzz-\uy|^{m-1}}d\uzz \leq M\epsilon \to 0 \quad(\epsilon\to 0).
\]
In a similar way, for $u\geq 1$
\[
	\Big|\int_{C^*_{\epsilon}(\uy)}E_{u}(\uzz-\uy)n(\uzz)P_{u+s}[f](\uzz)d\uzz\Big|\leq
	M\int_{C^*_{\epsilon}(\uy)}\frac{|\uzz-\uy|}{|\uzz-\uy|^{m-u-1}}d\uzz
	\leq M	\epsilon^{u+1}
\]
and
\[
	\Big|\int_{C^*_{\epsilon}(\uy)}E_{u}(\uzz-\uy)n(\uzz)\cD_{\uy}^{u+s}\f(\uy)d\uzz\Big|\leq
	M\int_{C^*_{\epsilon}(\uy)}\frac{d\uzz}{|\uzz-\uy|^{m-u-1}}
	\leq M\epsilon^{u}
\]
tend to zero as $\epsilon$ tends to $0$.
	
Then for $1\leq u\leq k-s$ we have by \eqref{P_a} that
\begin{align*}
		&\Big|\int_{C^*_{\epsilon}(\uy)}E_{u}(\uzz-\uy)n(\uzz)\cD_{\uzz}^{u}F(\uzz)d\uzz\Big|\\
		&=\Big|\int_{C^*_{\epsilon}(\uy)}E_{u}(\uzz-\uy)n(\uzz)[P_{u+s}[f](\uzz)+\cD_{\uy}^{u+s}\f(\uy)]d\uzz\Big|\to 0 \quad (\epsilon\to 0).
\end{align*}
Hence $K=0$ as claimed. In the same manner we can see that \eqref{lemaSj_a} holds for $s$ odd. 
\end{proof}

We are now in a position to state our main result, which ensures that the singular integral operator behaves as an involution on the higher order Lipschitz class.
\begin{theorem}\label{inv}
Let $f\in\Li(k+\alpha,\Gamma)$. Then $\cS_k^2=I$, where $I$ is the identity operator. That is,  
\begin{equation}\label{id_inv}
		[\cS_k^2f]^{(j)}=f^{(j)} \quad \forall\quad |j|\leq k.
\end{equation}
\end{theorem}

\begin{proof} Before going to the proof, we briefly outline our strategy. We first show \eqref{id_inv} for $j=0$ and then we take advantage of Theorem \ref{thm_fj} to show \eqref{id_inv} for all $|j|\leq k.$
	
Let $j=0$, then	
\begin{align*}
		[\cS_k^2f]^{(0)}(\uz)
		&=2\int_\Gamma E_0(\uy-\uz)n(\uy)\widetilde{\cS_kf}(\uy)d\uy\\
		&+2\sum\limits_{2\leq s-even}\int_\Gamma E_s(\uy-\uz)n(\uy)c_s\sum\limits_{\substack{|j|=s\\ j_i-even}} \partial^{(j)}_{\uy}\widetilde{\cS_kf}(\uy) d\uy+\\
		&+2\sum\limits_{s-odd}\int_\Gamma -E_s(\uy-\uz)n(\uy)c_s\sum\limits_{\substack{|j+l|=s\\j_i-even\\|l|=1}}e_{(l)}\partial^{(j+l)}_{\uy}\widetilde{\cS_kf}(\uy)d\uy.
\end{align*}	
Combining the fact that $\partial_{\uy}^{(j)}\widetilde{\cS_kf}\mid_{\Gamma}=[\cS_kf]^{(j)}$ with Lemma \ref{lemmaSj} we get
\begin{align*}
		[\cS_k^2f]^{(0)}(\uz)
		=2\int_\Gamma E_0(\uy-\uz)n(\uy)\Big[2\sum\limits_{u=0}^{k}\int_\Gamma (-1)^uE_u(\uzz-\uy)n(\uzz)\cD_{\uzz}^u\f(\uzz)d\uzz\Big]d\uy\\
		+2\sum\limits_{s=1}^k \int_\Gamma (-1)^s E_s(\uy-\uz)n(\uy) \Bigg\{2\sum\limits_{u=1}^k
		\int_\Gamma (-1)^u\cD_{\uy}^sE_u(\uzz-\uy)n(\uzz)\cD_{\uzz}^u \f(\uzz)d\uzz\Bigg\}d\uy.
\end{align*}	
The involution property for the classical singular operator yields 
\[
	2\int_\Gamma E_0(\uy-\uz)n(\uy)\Big[2\int_\Gamma E_0(\uzz-\uy)n(\uzz)\f(\uzz)d\uzz\Big]d\uy=f^{(0)}(\uz).
\]
Therefore, by \eqref{eq_Ds} and Theorem \ref{thm_whit}, we are reduce to proving that
\[
	2\sum\limits_{s=0}^k \int_\Gamma (-1)^s E_s(\uy-\uz)n(\uy) \Big[2\sum\limits_{u=1}^{k}\int_\Gamma (-1)^u\cD_{\uy}^sE_u(\uzz-\uy)n(\uzz)\cD_{\uzz}^u \f(\uzz)d\uzz\Big] d\uy=0.
\]
By Fubini's theorem this equation can be rewritten in the form:
\begin{equation}\label{eq_def_Q}
		4\sum\limits_{s=0}^{k}(-1)^s\sum\limits_{u=1}^{k}(-1)^u\int_\Gamma\Big[ \int_\Gamma E_s(\uy-\uz)n(\uy)\cD_{\uy}^sE_u(\uzz-\uy)d\uy\Big]n(\uzz)\cD_{\uzz}^u\f(\uzz)d\uzz=0.
\end{equation}
It will thus be sufficient to show that
\begin{equation}\label{TIC}
	Q(\uz):=\sum\limits_{s=0}^{k}\int_\Gamma (-1)^sE_s(\uy-\uz)n(\uy)\cD_{\uy}^s F(\uy)d\uy=0,
\end{equation}	
where $F(\uy):=\sum\limits_{u=1}^{k}(-1)^u E_u(\uzz-\uy)$. \\
The function $Q$ can be handled in much the same way as $K$ in the proof of Lemma \ref{lemmaSj}, the main difference being in the definition of $\Gamma^*$ and, hence, the estimations. We choose $\epsilon>0$ so small that the balls $B_\epsilon(\uz)$ and $B_\epsilon(\uzz)$ are disjoint. Set $\Gamma_{\epsilon}:=\Gamma\setminus\Gamma\cap \big(B_\epsilon(\uz)\cup B_\epsilon(\uzz)\big)$ so that
\[
	p.v. \int_\Gamma E_s(\uy-\uz)n(\uy)\cD_{\uy}^s F(\uy)d\uy=\lim_{\epsilon\to 0}\int_{\Gamma_\epsilon} E_s(\uy-\uz)n(\uy)\cD_{\uy}^s F(\uy)d\uy.
\]
Denote by $\Omega_*$ the interior of $\Gamma^*:=\Gamma_{\epsilon}\cup C^*_{\epsilon}(\uz)\cup C^*_{\epsilon}(\uzz)$. The last two sets being the intersections of $\Omega$ and the boundaries of those above-mentioned balls respectively. For abbreviation, we set $C_\epsilon^*=C^*_{\epsilon}(\uz)\cup C^*_{\epsilon}(\uzz)$. Then  $F(\uy)$ is $(k+1)$-monogenic in $\Omega_*$ and continuous on $\Gamma^*\cup\overline{\Omega}_*$. By the Cauchy formula (in $\R^m\setminus\overline{\Omega}_*$) it will be sufficient to prove that
\begin{align}\label{lim}
		\lim\limits_{\epsilon\to 0}
		\Bigg\{\sum\limits_{s=0}^{k}\int_{ C^*_{\epsilon}} (-1)^sE_s(\uy-\uz)n(\uy)\cD_{\uy}^s F(\uy)d\uy\Bigg\}=0.
\end{align}
Since $\cD_{\uy}^sF(\uy)=\sum\limits_{u=s}^{k}(-1)^{u+s} E_{u-s}(\uzz-\uy)=\sum\limits_{u=0}^{k-s}(-1)^{u} E_{u}(\uzz-\uy)$ for $s\geq 1$, the sum of integrals in \eqref{lim} becomes
\begin{align}\label{suma}
		&\sum\limits_{s=0}^{k}\int_{ C^*_{\epsilon}(\uz)\cup C^*_{\epsilon}(\uzz)} (-1)^sE_s(\uy-\uz)n(\uy)\cD_{\uy}^s F(\uy)d\uy\notag\\
		&=\sum\limits_{s=1}^{k}(-1)^s\int_{C^*_{\epsilon}} E_0(\uy-\uz)n(\uy)E_s(\uzz-\uy)d\uy\notag\\
		&+\sum\limits_{s=1}^{k}(-1)^s\int_{C^*_{\epsilon}} E_s(\uy-\uz)n(\uy)E_0(\uzz-\uy)d\uy \notag\\
		&+\sum\limits_{s=1}^{k-1}(-1)^s\int_{C^*_{\epsilon}} E_s(\uy-\uz)n(\uy)\sum\limits_{u=1}^{k-s}(-1)^uE_u(\uzz-\uy)d\uy.
\end{align}	
First, we note that the last term of \eqref{suma} splits into the pair of sums
\begin{align}\label{doblesuma}
		\sum\limits_{s=1}^{k-1}\sum\limits_{u=1}^{k-s}\int a_{su}=\sum\limits_{u=1}^{\lfloor \frac{k}{2} \rfloor}\int a_{uu}+\sum\limits_{\substack{s,u=1\\ s\neq u}}^{k-1}\int(a_{su}+a_{us}), k>1.
\end{align}
We divide the second sum in \eqref{doblesuma} into three cases depending on the parity: when both $s$ and $u$ are even ($s,u\geq 2$), when both $s$ and $u$ are odd and, finally, when $u$ is even ($u\geq 2$) and $s$ is odd. We only show the first one since similar arguments apply to the others. 
\begin{align}\label{Caso_A}
		\int_{C^*_{\epsilon}}\Big[E_s(\uy-\uz)n(\uy)E_u(\uzz-\uy)+E_u(\uy-\uz)n(\uy)E_s(\uzz-\uy)\Big]d\uy\notag\\
		=M\left\{\frac{1}{\epsilon^{m-s-1}}\int_{C^*_{\epsilon}(\uz)}\frac{\uy-\uzz}{|\uzz-\uy|^{m-u}}d\uy+\frac{1}{\epsilon^{m-u-1}}\int_{C^*_{\epsilon}(\uzz)}\frac{\uy-\uz}{|\uy-\uz|^{m-s}}d\uy\right.\notag\\
		+\left.\frac{1}{\epsilon^{m-u-1}}\int_{C^*_{\epsilon}(\uz)}\frac{\uy-\uzz}{|\uzz-\uy|^{m-s}}d\uy+\frac{1}{\epsilon^{m-s-1}}\int_{C^*_{\epsilon}(\uzz)}\frac{\uy-\uz}{|\uy-\uz|^{m-u}}d\uy\right\},
\end{align}	
where $M$ is a constant dependent on $s, u$ and $m$ throughout \cite{Ry}. After the change of variable $\uy= -\underline{t}+\uz+\uzz$, we get that the first and fourth integral in \eqref{Caso_A} are equal, and thus we are reduced to proving that its limit is zero as $\epsilon\to0$. This observation applies to the second and third integral as well. It is clear that $h(\uy):=\frac{\uy-\uzz}{|\uzz-\uy|^{m-u}}$ is continuous on $C^*_{\epsilon}(\uz)$ and thus bounded, which yields
\[
	\Big|\frac{1}{\epsilon^{m-s-1}}\int_{C^*_{\epsilon}(\uz)}\frac{\uy-\uzz}{|\uzz-\uy|^{m-u}}d\uy\Big|\leq \frac{N}{\epsilon^{m-s-1}}\int_{C^*_{\epsilon}(\uz)}d\uy=N'\epsilon^s\to 0 \quad (\epsilon\to 0).
\]
Hence the second sum in \eqref{doblesuma} tends to zero as $\epsilon\to 0$, and so does the first one. 
	
Finally, we note that the estimations of those integrals in the first two sums of \eqref{suma} whenever $s$ is odd and $u=0$ hold from the third case above-mentioned when we drop the assumption $u\geq 2$. It remains the case when $s$ is even (and $u=0$). That is,
\begin{align*}
		\int_{C_\epsilon^*}\Big[ E_0(\uy-\uz)n(\uy)E_s(\uzz-\uy)+ E_s(\uy-\uz)n(\uy)E_0(\uzz-\uy)\Big]d\uy\\
		=\int_{C^*_{\epsilon}(\uz)}E_0(\uy-\uz)n(\uy)E_s(\uzz-\uy)d\uy+\int_{C^*_{\epsilon}(\uzz)}E_0(\uy-\uz)n(\uy)E_s(\uzz-\uy)d\uy\\
		+\int_{C^*_{\epsilon}(\uz)}E_s(\uy-\uz)n(\uy)E_0(\uzz-\uy)d\uy+\int_{C^*_{\epsilon}(\uzz)}E_s(\uy-\uz)n(\uy)E_0(\uzz-\uy)d\uy.
\end{align*}
The first and fourth integral above are the left and right Cauchy transform in $\uy=\uz$ and $\uy=\uzz$ respectively of the functions $E_s(\uzz-\uy)$ and $-E_s(\uy-\uz)$, which are continuous in $B^*_{\epsilon}(\uz)$ and $B^*_{\epsilon}(\uzz)$ respectively. Letting $\epsilon\to0$ yields
\[
	\lim\limits_{\epsilon\to 0}\Bigg\{\int_{C^*_{\epsilon}(\uz)}E_0(\uy-\uz)n(\uy)E_s(\uzz-\uy)d\uy
	+\int_{C^*_{\epsilon}(\uzz)}E_s(\uy-\uz)n(\uy)E_0(\uzz-\uy)d\uy\Bigg\}=0.
\]
For the second and third integral, we proceed as in the first case and so $Q=0$ as desired.

We are now ready to proceed to the final stage of our proof. From Theorem \ref{thm_fj}, we only need to show that
\begin{enumerate}[(A)]
\item $\sum\limits_{\substack{|j|=s\\ j_i-even}} [\cS_k^2f]^{(j)}-\sum\limits_{\substack{|j|=s\\ j_i-even}} f^{(j)}=0$ if $s$ is even.\label{A}
\item $\sum\limits_{\substack{|j+l|=s\\j_i-even\\|l|=1}}e_{(l)}[\cS_k^2f]^{(j+l)}-\sum\limits_{\substack{|j+l|=s\\j_i-even\\|l|=1}}e_{(l)}f^{(j+l)}=0$ if $s$ is odd.\label{B}
\end{enumerate}
Let us first prove \ref{A}. After applying Lemma \ref{lemmaSj} twice and having in mind Lemma \ref{D_multii} and Theorem \ref{thm_whit} we obtain
\begin{align*}
		c_s	\sum\limits_{\substack{|j|=s\\ j_i-even}} \cS_k^{(j)}[\cS_k f](\uz)
		=2\sum\limits_{u=0}^{k-s}
		\int_\Gamma (-1)^u E_u(\uy-\uz)n(\uy)\cD_{\uy}^{u+s} [\widetilde{\cS_kf}](\uy)d\uy\notag\\
		=2\sum\limits_{u=0}^{k-s}
		\int_\Gamma (-1)^u E_u(\uy-\uz)n(\uy)\Big\{2\sum\limits_{v=0}^{k-(u+s)}\int_\Gamma (-1)^vE_v(\uzz-\uy)n(\uzz)\cD_{\uzz}^{v+(u+s)} \f(\uzz)d\uzz\Big\}d\uy.
\end{align*}
Since
\[
	2\int_\Gamma E_0(\uy-\uz)n(\uy)\Big\{2\int_\Gamma E_0(\uzz-\uy)n(\uzz)\cD_{\uzz}^s \f(\uzz)d\uzz\Big\}d\uy=c_s\sum\limits_{\substack{|j|=s\\ j_i-even}} f^{(j)}(\uz),
\]
the assertion follows if we show that $J$ vanishes, where 
\begin{align}\label{eq_def_J}
		J(\uz):=2\int_\Gamma E_0(\uy-\uz)n(\uy)\Big\{2\sum\limits_{v=1}^{k-s}\int_\Gamma (-1)^vE_v(\uzz-\uy)n(\uzz)\cD_{\uzz}^{v+s} \f(\uzz)d\uzz\Big\}d\uy\\
		+	2\sum\limits_{u=1}^{k-s}
		\int_\Gamma (-1)^u E_u(\uy-\uz)n(\uy)\Big\{2\sum\limits_{v=0}^{k-(u+s)}\int_\Gamma (-1)^vE_v(\uzz-\uy)n(\uzz)\cD_{\uzz}^{v+(u+s)} \f(\uzz)d\uzz\Big\}d\uy\notag.
\end{align}
We check at once that \eqref{eq_def_J} can be rewritten as
\[
	J(\uz)=2\sum\limits_{u=0}^{k-s}\int_\Gamma (-1)^u E_u(\uy-\uz)n(\uy)\Big\{2\sum\limits_{v=1}^{k-s}\int_\Gamma (-1)^{v}\cD_{\uy}^uE_{v}(\uzz-\uy)n(\uzz)\cD_{\uzz}^{v+s}\f(\uzz)d\uzz\Big\}d\uy,
\]
which is clear from the multiplication rules for the kernels $E_v$.

A similar approach to that in the proof of \eqref{eq_def_Q} shows that $J=0$. This time the function $F(\uy)=\sum\limits_{v=1}^{k-s}(-1)^{v}E_{v}(\uzz-\uy)$ being $(k-s+1)$-monogenic instead. In the same manner we can see that case \ref{B} holds, which completes the proof.
\end{proof}

We note that Theorem \ref{inv} extends \cite[Thm. 3]{DAB2}. From what has already been proved, it is easy to see that the operators $\cP^+=\frac{1}{2}(I+\cS_k)$ and $\cP^-=\frac{1}{2}(I-\cS_k)$
are projections on $\Li(k+\alpha,\Gamma)$, that is 
\[
\cP^+\cP^+=\cP^+,\quad\cP^-\cP^-=\cP^-,\quad\cP^+\cP^-=0=\cP^-\cP^+.
\]
Consequently, the Hardy decomposition follows:
\[\Li(k+\alpha,\Gamma)=\Li^+(k+\alpha,\Gamma)\oplus\Li^-(k+\alpha,\Gamma),\]
where $\Li^\pm(k+\alpha,\Gamma):=\mbox{im}\cP^\pm$. We now turn to characterize $\Li^\pm(k+\alpha,\Gamma)$. 

\begin{theorem} \label{Thm_Lip_plus}
The Whitney data $f\in\emph{\Li}(k+\alpha,\Gamma)$ belongs to $\emph{\Li}^+(k+\alpha,\Gamma)$ if and only if there exists a $(k+1)$-monogenic function $F$ in $\Omega_+$, which together with $\cD_{\ux}^uF$, $u=\overline{0,k}$ continuously extends to $\Gamma$ and such that
\begin{equation}\label{trace+}
		F|_\Gamma =f^{(0)}, \cD_{\ux}^uF|_\Gamma =c_s\begin{cases}
			\sum\limits_{\substack{|j|=s\\ j_i-even}} f^{(j)}, &$s$ \quad\emph{even},\\
			\sum\limits_{\substack{|j+l|=s\\j_i-even\\|l|=1}}e_{(l)}f^{(j+l)},&$s$ \quad\emph{odd}.
		\end{cases}
\end{equation}
\end{theorem}

\begin{proof}
The proof is similar in spirit to \cite[Thm 3]{DAB2}, but here the key point is Lemma \ref{lemmaSj}. We begin by proving the necessity. By definition, if $f\in\Li^+(k+\alpha,\Gamma)$ there exists $g\in\Li(k+\alpha,\Gamma)$ such that $f=\frac{1}{2}(g+\cS_kg)$, i.e.  
$f^{(j)}(\ux)=\frac{1}{2}[I^{(j)}+\cS_k^{(j)}]g, |j|\leq k$, where $I^{(j)}g=g^{(j)}$ is the identity operator in $\Li(k-|j|+\alpha,\Gamma)$. Let us introduce the  function $F$ given by the Cauchy type transform $F(\ux)=\cC_k^{(0)}g(\ux), \ux\in\Omega_+$, which is $(k+1)$-monogenic in $\Omega_+$. Thus, for $\uz\in\Gamma$ we get from \eqref{cliff_Plemelj-Sokhotski_Ck} that
\[
	F(\uz)=\lim_{\substack{\ux\to\uz\\ \ux\in\Omega_+}}F(\ux)=
	[\cC_k^{(0)}g]^+(\uz)=\frac{1}{2}[I^{(0)}+\cS_k^{(0)}]g=f^{(0)}(\uz).
\]
On the other hand, for $s$ even
\[
	\cD_{\ux}^sF(\ux)=\sum_{v=0}^{k-s}\int_\Gamma (-1)^v E_v(\uy-\ux)n(\uy)\cD_{\uy}^{v+s}\g(\uy)d\uy=\cC_{k-s}^{(0)}[\cD_{\ux}^s\g\mid_{\Gamma}](\ux).
\]
By \eqref{cliff_Plemelj-Sokhotski_Ck} again and the fact that $\cD_{\ux}^s\g\mid_{\Gamma}\in \Li(k-s+\alpha,\Gamma)$ we have
\begin{align*}
		&\cD_{\ux}^sF(\uz)
		=\Big[\cC_{k-s}^{(0)}(\cD_{\ux}^s\g\mid_{\Gamma})\Big]^+(\uz)\\
		&=\frac{1}{2}\Big[c_s\sum\limits_{\substack{|j|=s\\ j_i-even}}g^{(j)}(\uz)+2\sum_{u=0}^{k-s}\int_\Gamma (-1)^u E_u(\uy-\uz)n(\uy)\cD_{\uy}^u\widetilde{\cD_{\uy}^s\g(\uy)}d\uy\Big] \\
		&=\frac{1}{2}\Big[c_s\sum\limits_{\substack{|j|=s\\ j_i-even}}g^{(j)}(\uz)+2\sum_{v=1}^{k}\int_\Gamma (-1)^v \cD_{\uy}^vE_v(\uy-\uz)n(\uy)\cD_{\uy}^v\g(\uy)d\uy\Big].
\end{align*}
Lemma \ref{lemmaSj} then leads to
\[
	\cD_{\ux}^sF(\uz)=c_s\sum\limits_{\substack{|j|=s\\ j_i-even}}\frac{1}{2}\Big[I^{(j)}+\cS_k^{(j)}\Big]g(\uz)=c_s\sum\limits_{\substack{|j|=s\\ j_i-even}}f^{(j)}(\uz)
\]
as desired. The same reasoning applies to the second statement of \eqref{trace+} when $s$ is odd.
	
We now turn to the sufficiency. Assume there exists such a $(k+1)$-monogenic function $F$ satisfying \eqref{trace+}. If we prove that $[\cP^+f]^{(j)}=f^{(j)}$ for all $0\le |j|\le k$, then $f\in\mbox{im}\cP^+$ and the assertion follows.
	
Let us apply \eqref{repr_dom_+} to $F$ and make use of \eqref{trace+} to get
\begin{align*}
		F(\ux)&=\int_\Gamma  E_0(\uy-\ux)n(\uy)f^{(0)}(\uy)d\uy+\sum_{\substack{u-even\\ u\geq 2}}\int_\Gamma E_u(\uy-\ux)n(\uy) c_u\sum\limits_{\substack{|j|=u\\ j_i-even}} f^{(j)}(\uy)d\uy\\
		&-\sum_{\substack{u-odd\\ u\geq 1}}\int_\Gamma E_u(\uy-\ux)n(\uy)c_u\sum\limits_{\substack{|j+l|=u\\j_i-even\\|l|=1}}e_{(l)}f^{(j+l)}(\uy)d\uy.
\end{align*}
Combining Lemma \ref{D_multii} and Theorem \ref{thm_whit} gives
\[
	F(\ux)=\sum_{u=0}^{k}\int_\Gamma (-1)^u E_u(\uy-\ux)n(\uy)\cD_{\uy}^u\f(\uy)d\uy:=\cC_k^{(0)}f(\ux).
\]
From this and \eqref{cliff_Plemelj-Sokhotski_Ck} we obtain
\[
	[\cP^+f]^{(0)}(\uz):=\frac{1}{2}[I^{(0)}+\cS_k^{(0)}]f(\uz)=[\cC_k^{(0)}f]^+(\uz)=F(\uz)=f^{(0)}(\uz).
\]
On the other hand, for $|j|=s$ with $s$ even we have from Lemma \ref{lemmaSj}-\eqref{lemaSj_b}
\begin{align}\label{Pjeven}
	c_s\sum\limits_{\substack{|j|=s\\ s-even}}[\cP^+f]^{(j)}(\uz):=c_s\sum\limits_{\substack{|j|=s\\ s-even}}\frac{1}{2}[I^{(j)}+\cS_k^{(j)}]f(\uz)\notag\\
		=\frac{1}{2}\Big[c_s\sum\limits_{\substack{|j|=s\\ s-even}}f^{(j)}(\uz)+2\sum\limits_{u=0}^{k-s}
		\int_\Gamma (-1)^uE_u(\uy-\uz)n(\uy)\cD_{\uy}^{u}\big(\cD_{\uy}^{s}\f(\uy)\big)d\uy\Big].
\end{align}
We note that according to \eqref{trace+}, the function $g:=\cD_{\uy}^s\f\mid_{\Gamma}=c_s\sum\limits_{\substack{|j|=s\\ s-even}}f^{(j)}(\uy)$ reresents the interior limiting value of the ($k-s+1$)-monogenic function $\cD_{\ux}^sF(\ux)$, that is 
\[
	g(\uz)=[\cC_{k-s}^{(0)}(\cD_{\uz}^sF)]^+(\uz),\quad\uz\in\Gamma.
\]
By \eqref{cliff_Plemelj-Sokhotski_Ck} we get more, namely $g(\uz)=\frac{1}{2}[I^{(0)}+S_{k-s}^{(0)}]g(\uz)$. Consequently,
\[
	\frac{1}{2}c_s\sum\limits_{\substack{|j|=s\\ s-even}}f^{(j)}(\uz)=\frac{1}{2}\cS_{k-s}^{(0)}\big(c_s\sum\limits_{\substack{|j|=s\\ s-even}}f^{(j)}\big)(\uz).
\]
When this is substituted in \eqref{Pjeven} we finally get
	\[c_s\sum\limits_{\substack{|j|=s\\ s-even}}[\cP^+f]^{(j)}(\uz)=c_s\sum\limits_{\substack{|j|=s\\ s-even}}f^{(j)}(\uz).\]
	The same conclusion can be drawn for $c_s\sum\limits_{\substack{|j+l|=s\\j_i-even\\|l|=1}}e_{(l)}[\cP^+f]^{(j+l)}(\uz)$ when $s$ is odd. Accordingly, by Theorem \ref{thm_fj} the proof is complete.
\end{proof}

For a characterization of the space $\Li^-(k+\alpha,\Gamma)$ we look at a representation formula of polymonogenic functions in the exterior domain. 
Assume $f\in \bC^k(\Omega_-)\cap \bC^{k-1}(\Omega_-\cup\Gamma)$ is $k$-monogenic in $\Omega_-$, $f(\infty)$ exists and $\cD_{\uy}^uf(\uy)=o\big(\frac{1}{|\uy|^u}\big)$ as $|\uy|\to\infty$ for every $u=\overline{0,k-1}$. Consider the ball $B_R(\ux)$ with center in $\ux\in\Omega_-$ and radius $R$ sufficiently large such that $\Omega_+\cup\Gamma\subset B_R(\ux)$. Then, if \eqref{repr_dom_+} is applied to the domain $B_R(\ux)\setminus \Omega_+\cup\Gamma$ we obtain
\[
f(\ux)=\sum_{u=0}^{k-1}\int_{\Gamma^*} (-1)^uE_u(\uy-\ux)n(\uy) \cD_{\uy}^uf(\uy)d\uy,
\]
where $\Gamma^*=-\Gamma\cup C_R(\ux)$ and $C_R(\ux)=\partial B_R(\ux)$.

Since $\int_{C_R(\ux)} E_{0}(\uy-\ux)n(\uy)d\uy=1$, the continuity of $f$ and the identity
\begin{align*}
	\int_{C_R(\ux)} E_{0}(\uy-\ux)n(\uy) f(\uy)d\uy=\int_{C_R(\ux)} E_{0}(\uy-\ux)n(\uy)[f(\uy)-f(\infty)]d\uy\\
	+\Big(\int_{C_R(\ux)} E_{0}(\uy-\ux)n(\uy)d\uy\Big)f(\infty),
\end{align*}
implies that 
\begin{equation*}
	\int_{C_R(\ux)} E_{0}(\uy-\ux)n(\uy) f(\uy)d\uy\to f(\infty),\quad \mbox{as}\quad\R\to \infty.
\end{equation*}
On the other hand,
\begin{align*}
	|\uy-\ux|^u \lVert\cD_{\uy}^uf(\uy)\rVert&\leq \sum\limits_{i=0}^{u}\binom{u}{i}|\uy|^i|\ux|^{u-i}\lVert\cD_{\uy}^uf(\uy)\rVert \\
	&\leq \sum\limits_{i=0}^{u}\binom{u}{i}|\ux|^{u-i}|\uy|^u\lVert\cD_{\uy}^uf(\uy)\rVert\quad \mbox{as}\quad |\uy|\to\infty.
\end{align*}
From the assumed behaviour at infinity it follows that $|\uy|^u\lVert\cD_{\uy}^u f(\uy)\rVert\leq \epsilon$ as $|\uy|\to \infty$ and thus $\lVert\cD_{\uy}^u f(\uy)\rVert\to 0$ as $|\uy|\to\infty$ for every $u=\overline{0,k-1}$. Consequently, $\cD_{\uy}^u f(\uy)=o\big(\frac{1}{|\uy-\ux|^u}\big)$ as $|\uy|\to \infty$. Therefore,
\begin{align*}
	\Big\lVert\sum_{u=1}^{k-1} \int_{C_R(\ux)} (-1)^uE_u(\uy-\ux)n(\uy) \cD_{\uy}^uf(\uy)d\uy \Big\rVert&\leq \sum_{u=1}^{k-1}\int_{C_R(\ux)} |E_u(\uy-\ux)|\lVert \cD_{\uy}^u f(\uy)\rVert d\uy\\
	&\leq \frac{c\epsilon}{R^{m-1}}\int_{C_R(\ux)}d\uy\to 0\,(\epsilon\to 0,\R\to\infty).
\end{align*}
Under the above-mentioned assumptions, we arrive at the following representation formula in the exterior domain:
\begin{align}\label{repr_dom_-}
	f(\ux)=-\sum_{u=1}^{k-1} \int_{\Gamma} (-1)^uE_u(\uy-\ux)n(\uy) \cD_{\uy}^uf(\uy)d\uy+f(\infty),\quad \ux\in\Omega_-.
\end{align}
\begin{theorem}\label{Thm_Lip_minus} 
The Whitney data $f\in\Li(k+\alpha,\Gamma)$ belongs to $\Li^-(k+\alpha,\Gamma)$ if and only if there exists a polymonogenic function $F$ in $\Omega_-$ vanishing at infinity satisfying $\cD_{\uy}^uf(\uy)=o\big(\frac{1}{|\uy|^u}\big)$ as $|\uy|\to\infty$ for every $u=0,\ldots,k-1$, which together with $\cD_{\uy}^uF$,  $u=0,\ldots,k-1$ continuously extends to $\Gamma$ and such that
\begin{equation*}%\label{trace-}
		F|_\Gamma =f^{(0)},\quad \cD_{\ux}^uF|_\Gamma =c_s
		\begin{cases}
			\sum\limits_{\substack{|j|=s\\ j_i-even}} f^{(j)}, &$s$\quad \emph{even},\\
			\sum\limits_{\substack{|j+l|=s\\j_i-even\\|l|=1}}e_{(l)}f^{(j+l)},&$s$\quad \emph{odd}.
		\end{cases}
\end{equation*}
\end{theorem}
Although Riemann-Hilbert problems (RH for short) for polymonogenic functions with boundary data given in $\mathbb{L}_p$ ($1<p<\infty$) or $\bC^\alpha$ spaces have been studied in the literature (see, e.g., \cite{CKK,HKSB,KFKC,YJ}), a RH for polymonogenic functions with boundary data in $\Li(k+\alpha,\Gamma)$ was first studied in \cite{AAB} on fractal domains. We close this article with our approach to the RH of finding a sectionally polymonogenic function $F$ satisfying the conditions
\[
\begin{cases}
	F^+(\uz)-F^-(\uz)=f^{(0)}(\uz),&\uz\in\Gamma\\
	[\cD_{\ux}^u F]^+(\uz)-[\cD_{\ux}^u F]^-(\uz)= c_s\sum\limits_{\substack{|j|=s\\ j_i-even}} f^{(j)}(\uz), &\uz\in\Gamma,\quad  s\quad \text{even}\\
	[\cD_{\ux}^u F]^+(\uz)-[\cD_{\ux}^u F]^-(\uz)= c_s\sum\limits_{\substack{|j+l|=s\\j_i-even\\|l|=1}}e_{(l)}f^{(j+l)}(\uz),&\uz\in\Gamma,\quad  s\quad \text{odd}\\
	\cD_{\ux}^uF=o\bigg(\dfrac{1}{|\ux|^u}\bigg),& {\mbox{as}}\quad  \ux\to\infty\\
	F(\infty)=0,
\end{cases}
\]
where $f$ is a given function of the class $\Li(k+\alpha,\Gamma)$. From what has already been shown, it follows that the unique solution to this problem is given by $F=\cC_k^{(0)}f.$

\section{Acknowledgements}
This research was completed when the first author was visiting the TU Bergakademie Freiberg under the auspices of the Alexander von Humboldt Foundation; their hospitality and founding respectively are gratefully acknowledged.


\begin{thebibliography}{99}	
\bibitem{AA} Abreu Blaya, R., Alonso Santiesteban, D., Bory Reyes, J., Moreno Garc\'ia, A.: Inframonogenic decomposition of higher order Lipschitz functions. \emph{Math. Methods Appl. Sci.} \textbf{45} (2022), 4911--4928, 124559.
\bibitem{AD} Abreu Blaya. R., De la Cruz Toranzo, L.: Polyanalytic Hardy decomposition of higher order Lipschitz functions. \emph{J. Math. Anal. Appl.} \textbf{493} (2021), 124559.
\bibitem{AAB} Abreu Blaya, R., \'Avila \'Avila, R., Bory Reyes, J.: Boundary value problems with higher order Lipschitz boundary data for polymonogenic functions in fractal domains. \emph{Appl. Math. Comput.} \textbf{269} (2015), 802--808.
\bibitem{Ba} Balk, M.~B.: \emph{On Polyanalytic Functions}. Akademie Verlag, Berlin, 1991.
\bibitem{BDS} Brackx, F., Delanghe, R., Sommen, F.: \emph{Clifford Analysis}. Boston, Pitman, 1982.
\bibitem{CKK} Cerejeiras, P., K\"ahler, U., Ku, M.: On the Riemann Boundary Value Problem for Null Solutions to Iterated Generalized Cauchy-Riemann Operator in Clifford Analysis. \emph{Results. Math.} \textbf{63} (2013), 1375--1394.
\bibitem{DAB1} De la Cruz Toranzo, L., Abreu Blaya, R., Bory Reyes, J.: On the Plemelj-Privalov theorem in Clifford analysis involving higher order Lipschitz classes. \emph{J. Math. Anal. Appl.} \textbf{480} (2019), 123411.
\bibitem{DAB2} De la Cruz Toranzo, L., Abreu Blaya, R., Bernstein, S.: Hardy decomposition of first order Lipschitz functions by Clifford algebra-valued harmonic functions. \emph{J. Math. Anal. Appl.} \textbf{536} (2024), 128242. 
\bibitem{GS} G\"urlebeck, K., Habetha, K., Spr\"ossig, W.: \emph{Holomorphic Functions in the Plane and n-dimensional Space}. Birkh\"auser, 2008.
\bibitem{HKSB} He, F., Ku, M., K\"ahler, U., Sommen, F., Bernstein, S.: Riemann-Hilbert problems for null-solutions to iterated generalized Cauchy-Riemann equations in axially symmetric domains. \emph{Comput. Math. Appl.} \textbf{71} (2016), 1990--2000.
\bibitem{Ift} Iftimie, V.: Fonctions hypercomplexes. \emph{Bull. Math. Soc. Sci. Math. R. S. Roumanie 9} \textbf{57} (1965), no.~4, 279--332.
\bibitem{KFKC} Ku, M., Fu, Y., K\"ahler, U., Cerejeiras, P.: Riemann boundary value problems for iterated Dirac operator on the ball in Clifford analysis. \emph{Complex Anal. Oper. Theory} \textbf{7} (2013), 673--693.
\bibitem{Mu} Mushelisvili, N.I.: \emph{Singular integral equations}. Nauka, Moscow, 1968.
\bibitem{Ry} Ryan, J.: Basic Clifford Analysis. \texttt{Cubo} \textbf{2} (2000), no.~1, 224--254.
\bibitem{St} Stein, E.M.: \emph{Singular Integrals and Differentiability Properties of Functions}. Princeton Math. Ser. 30, Princeton Univ. Press, Princeton N.J., 1970.
\bibitem{Wh} Whitney, H.: Analytic extensions of differentiable functions defined in closed sets. \emph{Trans. Amer. Math. Soc.} \textbf{36} (1934), no.~1, 63--89.
\bibitem{YJ} Yude, B., Jinyuan, D.: The RH boundary value problem of the $k$-monogenic functions. \emph{J. Math. Anal. Appl.} \textbf{347} (2008), 633--644.
\end{thebibliography}
\end{document}